\documentclass[11pt,reqno,oneside]{amsart}
 \usepackage{amsmath,amssymb,cite,epsfig,xspace,alltt,verbatim,bm}
\numberwithin{equation}{section}

\theoremstyle{plain}
\newtheorem{theorem}{Theorem}[section]
\newtheorem{lemma}{Lemma}[section]

\theoremstyle{definition}

\theoremstyle{remark}
\newtheorem{remark}{Remark}[section]

\newcommand{\eps}{\epsilon}

\newcommand{\Real}{\mathbb R}
\newcommand{\mvz}{\Real^{2N}_*}

\newcommand{\eb}{{\mathbf e}}

\newcommand{\fb}{{\mathbf f}}

\newcommand{\rb}{{\mathbf r}}

\newcommand{\ub}{{\mathbf u}}
\newcommand{\etab}{{\bm \eta}}
\newcommand{\xib}{{\bm \xi}}
\newcommand{\vb}{{\mathbf v}}
\newcommand{\wb}{{\mathbf w}}
\newcommand{\yb}{{\mathbf y}}

\newcommand{\Ds}{\mathcal{V}_0}

\newcommand{\pd}[2]{\frac{\partial #1}{\partial #2}}

\newcommand{\E}{{\mathcal E}}
\newcommand{\R}{{\mathcal R}}

\newcommand{\lpnorm}[2]{\left\|#1\right\|_{\ell^{#2}_\eps}}

\begin{document}

\title[Stability, Instability, and Error of the Force-based QC Approximation]{Stability, Instability, and Error of the \\
  Force-based Quasicontinuum Approximation }

\author{Matthew Dobson}
\author{Mitchell Luskin}
\author{Christoph Ortner}

\address{Matthew Dobson\\
School of Mathematics \\
University of Minnesota \\
206 Church Street SE \\
Minneapolis, MN 55455 \\
U.S.A.}
\email{dobson@math.umn.edu}

\address{Mitchell Luskin \\
School of Mathematics \\
University of Minnesota \\
206 Church Street SE \\
Minneapolis, MN 55455 \\
U.S.A.}
\email{luskin@umn.edu}

\address{Christoph Ortner\\
Mathematical Institute \\
University of Oxford \\
24-29 St Giles'\\
Oxford OX1 3LB \\
UK}
\email {ortner@maths.ox.ac.uk}

%%%%%%%%%%%%%%%%%%%
% Acknowledgments %
%%%%%%%%%%%%%%%%%%%
\thanks{
This work was supported in part by DMS-0757355,
 DMS-0811039,  the Department of Energy under Award Number
DE-FG02-05ER25706, the Institute for Mathematics and
Its Applications,
 the University of Minnesota Supercomputing Institute,
the University of Minnesota Doctoral Dissertation Fellowship,
and the EPSRC critical mass programme ``New Frontier in
the Mathematics of Solids.''
}

%%%%%%%%%%%%
% Keywords %
%%%%%%%%%%%%

%%%%%%%%%%%%%%%%%
% Subject class %
%%%%%%%%%%%%%%%%%

%%%%%%%%
% Date %
%%%%%%%%
\date{\today}
\begin{abstract}
  Due to their algorithmic simplicity and high accuracy, force-based
  model coupling techniques are popular tools in computational
  physics.  For example, the force-based quasicontinuum approximation
  is the only known pointwise consistent quasicontinuum (QC)
  approximation for coupling a general atomistic model with a finite
  element continuum model.  In this paper, we present a detailed
  stability and error analysis of this method.  Our optimal order
  error estimates provide a theoretical justification for the high
  accuracy of the force-based QC approximation: they clearly
  demonstrate that the computational efficiency of continuum modeling
  can be utilized without a significant loss of accuracy if defects
  are captured in the atomistic region.

  The main challenge we need to overcome is the fact (which we prove)
  that the linearized QC operator is typically {\em not} positive
  definite. Moreover, we prove that no uniform inf-sup stability
  condition holds for discrete versions of the $W^{1,p}$-$W^{1,q}$
  ``duality pairing'' with $1/p+1/q=1$, if $1 \leq p < \infty$.  We
  therefore derive an inf-sup stability condition for a discrete
  version of the $W^{1,\infty}$-$W^{1,1}$ ``duality pairing'' which
  then leads to optimal order error estimates in a discrete
  $W^{1,\infty}$-norm.
\end{abstract}

{

\maketitle
\thispagestyle{empty}
\section{Introduction}
Localized defects in materials typically interact with elastic fields
far beyond the defects' atomic neighborhood.  Accurately computing the
structure of localized defects requires the use of atomistic models;
however, atomistic models are too computationally demanding to be
utilized for the entire interacting system.  The goal of atomistic to
continuum coupling methods such as the quasicontinuum (QC) method is
to use the computationally intensive, fully atomistic calculations
only in regions with highly non-uniform deformations such as
neighborhoods of dislocations, crack tips, and grain boundaries; and
to use a (local) continuum model in regions with nearly uniform
deformations to reduce the number of degrees of freedom.

The initial computational results obtained with the QC method have
excited the materials science community with the promise of the
simulation of heretofore inaccessible multiscale materials problems
\cite{tadmor_qc_first,tadmor_miller_qc_overview,knap2003}.  Variants
of the QC method have continued to be developed with the introduction
of adaptive methods, improved mesh generation, and faster solvers
\cite{gaviniorbital,lu06,arndtluskin07c,dobsonluskin08,arndtluskin07b,PrudhommeBaumanOden:2005,shenoy_gf};
yet, in common with many other multiscale methods it lacks the
theoretical basis needed to give predictive computational results.

During the past few years, a mathematical structure has been given to
the description and analysis of various flavors of the QC method,
clarifying the relation between different approximations and the
corresponding sources of
error
~\cite{dobsonluskin07,BadiaParksBochevGunzburgerLehoucq:2007,minge05,pinglin03,Gunzburger:2008a,Gunzburger:2008b,emingyang,dobsonluskin09,pinglin05,ortnersuli,ezhang06,luskinortner_cluster}. In
the present paper, we contribute to this effort by providing a detailed
stability and error analysis of the {\em force-based quasicontinuum
approximation}.

Considerable concern has been generated by the discovery that early QC
approximations exhibit ``ghost'' forces in the atomistic to continuum
interface when the material is subject to a uniform strain, that is,
they do not satisfy the ``patch test'' criterion of computational
mechanics.  The first remedy, which is still commonly employed, is
known as the ghost force correction.  It first applies a dead load
that corrects the ghost forces at the current state of a continuation
process, then increments the parameter value for the process, then
reminimizes the energy at the new parameter value and dead load, and
finally recomputes the ghost force corrections for use as a dead load
at the next step of the continuation
process~\cite{shenoy_gf,tadmor_miller_qc_overview}.  In
\cite{dobsonluskin07,dobsonluskin08} this process was identified as an
iterative method to approximate the solution to
the equilibrium equations for a purely force-based
coupling approach, which we label the {\em force-based quasicontinuum
  (QCF) approximation}. This formulation of the QCF method has enabled
the development of more efficient iterative and continuation methods
for its solution and a more precise understanding of the
error~\cite{dobsonluskin07,dobsonluskin08}. Related force-based
modeling approaches, which couple an atomistic region with a continuum
region modeled by linear elasticity can be found
in~\cite{kohlhoff,shilkrot}.

Other research groups have proposed QC approximations that utilize
special interfacial atoms at the atomistic to continuum interface in
an attempt to develop an energy-based QC method that does not suffer
from the ghost force problem mentioned above~\cite{ezhang06,shim04}.
The quasi-nonlocal (QNL) approach~\cite{shim04} is easy to implement
and removes ghost forces for short range interactions (depending on
the lattice structure), but ghost forces remain for longer range
interactions.  This method was generalized in the reconstruction
approach~\cite{ezhang06} which, in theory, allows for the elimination
of all ghost forces; however, explicit methods have only been
constructed for planar interfaces so far. Moreover, a computationally
efficient implementation of this method, that can be used with
adaptive atomistic to continuum modeling algorithms, has yet to be
proposed.

Both of the above methods~\cite{ezhang06,shim04} couple the original
atomistic model to a new atomistic model with local interactions.  To
allow for the reduction of degrees of freedom by piecewise linear
interpolation in the continuum region as in the finite element method,
it is necessary to further couple this local atomistic model to a
volume-based local model.  However, it is not known how to couple a
local atomistic model to a volume-based local model along a nonplanar
interface without introducing ghost forces~\cite{ezhang06}.  In
contrast, the force-based quasicontinuum approximation allows
arbitrary atomistic to continuum interfaces and coarsening without
ghost forces.

Rather than computing forces from a total energy, the force-based
quasicontinuum approximation directly assigns forces using a simple
rule: {\em the force on an atom in the atomistic region is computed
  from the force law of the atomistic model, while the force on a
  degree of freedom in the continuum region is computed from the force
  law of the continuum (finite element)
  approximation}~\cite{dobsonluskin07,dobsonluskin08}.  There is no
modification of these equations near the atomistic to continuum
interface and it is therefore easy to see that the QCF equilibrium
equations are satisfied exactly by a material under uniform strain,
that is, there are no ghost forces in this approximation. Moreover, we
will show below that the QCF approximation has an O($\eps^2$)
truncation error in the atomistic to continuum interface for all
smoothly varying strains.  By contrast, it has been shown
in~\cite{dobsonluskinopt} that, even when it succeeds in removing
ghost forces, the QNL method has an O($1$) truncation error in the
atomistic to continuum interfaces for a nonuniform but smooth strain.
(This is nevertheless a significant improvement over the O($1/\eps$)
truncation error in the original QC method.)

With the exception of \cite{ortnersuli}, error analyses of
energy-based QC methods have utilized the coercivity
(positive-definiteness) of the linearization of the quasicontinuum
equilibrium equations about the energy-minimizing
solution~\cite{emingyang,dobsonluskinopt,minge05}.  A recent attempt
to establish an error analysis for the QCF method has presented an
invalid proof of coercivity of the linearized equilibrium equations
and an error analysis based on this incorrect coercivity
result~\cite{ming_gf}.  In the present paper we prove that, typically,
the linearization of the QCF equilibrium equations is {\em not}
coercive (cf. Theorem \ref{th:nocoerc}), and consequently, our error
analysis will be based on a more general inf-sup stability
condition. However, even this more general approach will fail unless
one chooses the norms particularly carefully. We show in
Section~\ref{sec:stability} that the linearized QCF operator is stable
with respect to a discrete version of the $W^{1,\infty}$-$W^{1,1}$
pairing, uniformly with respect to the number of atoms, but we show in
Section~\ref{sec:norms} that it is {\em not} uniformly stable with
respect to any other $W^{1,p}$-$W^{1,q}$ pairing where $\frac{1}{p} +
\frac{1}{q} = 1.$

Our goal in this paper is to clearly present our techniques in the
simplest setting.  For this reason, we restrict our presentation to a
one-dimensional chain of atoms which interact with nearest and
next-nearest neighbors.  To further simplify the setting, we consider
a linearization of the force-based equilibrium equations about a
uniform strain.  Although the QCF approximation can be directly
formulated and implemented with mesh coarsening in the continuum
region, we only consider the modeling error due to the QCF
approximation itself and do not consider the coarsening error.  Each
of these extensions deserve a careful analysis in order to firmly
establish the mathematical foundation of the force-based
quasicontinuum approximation.

The main result of the present paper is that the strain error for the
QCF approximation is O($\eps^2$), where $\eps$ is the lattice spacing
scaled by the material dimension. The prefactor for the $\eps^2$ error
term is a maximum norm of a third divided difference of the displacement
restricted to the continuum region only.  Thus, our analysis predicts
the observed high accuracy of the QCF method when defects are modeled
in the atomistic region.

In Section~\ref{sec:model}, we present a detailed description of the
force-based quasicontinuum approximation and a first estimate of the
truncation error with respect to the fully atomistic model.  In
Section~\ref{sec:conjugate}, we show how to formulate the QCF
approximation in a ``weak form'' that allows us to study its stability
by considering discrete versions of the $W^{1,p}$-$W^{1,q}$ ``duality
pairing.''  This is equivalent to putting the QCF operator into a
divergence form, which will indicate an interesting nonlocal effect of
the atomistic to continuum interface. This nonlocal effect is the
source of the lack of coercivity which we establish in
Section~\ref{sec:nocoerc}, based on the explicit construction of an
unstable displacement.  In Section~\ref{sec:stability}, we derive
inf-sup stability results that are then combined, in
Section~\ref{sec:converge}, with negative-norm truncation error
estimates, to obtain optimal order error estimates for the force-based
quasicontinuum approximation.  We conclude, in
Section~\ref{sec:norms}, by showing the lack of a uniform inf-sup
constant for all other common choices of duality pairings.

\section{The Force-Based Quasicontinuum Approximation}
\label{sec:model}

We consider a one-dimensional atomistic chain whose $2M+1$ atoms
occupy the reference positions $x_j = j\eps $, where $\eps$ is the
atomic spacing in the reference configuration, and which interact with
their nearest and next-nearest neighbors. We denote the deformed
positions by $y_j$ for $j=-M,\dots,M$. The boundary atoms are constrained by
\begin{equation*}
y_{-M}=-FM\eps\quad\text{and}\quad y_M=FM\eps,
\end{equation*}
where $F>0$ is a macroscopic deformation gradient.  The total energy
of a deformation ${\bf y} \in \mathbb{R}^{2N+1}$ is given by
\begin{equation}\label{atoma}
  \E^a(\yb) -\sum_{j=-M}^{M}\eps f_j y_j,
\end{equation}
where
\begin{equation}\label{atomaa}
\E^a(\yb)=
\sum_{j=-M+1}^{M} \eps \phi\Big(\frac{y_{j} - y_{j-1}}{\eps} \Big)
+ \sum_{j = -M+2}^{M} \eps \phi\Big(\frac{y_{j} - y_{j-2}}{\eps}\Big)
\end{equation}
for a scaled two-body interatomic potential $\phi$ (for example,
the normalized Lennard-Jones potential $\phi(r)
= \eps^{12} r^{-12}-2 \eps^6 r^{-6}$)
and external forces $f_j$.  The equilibrium equations are given by the
force balance at the free atoms,
\begin{equation}
\label{full}
\begin{aligned}
F_j^{a}(\yb) + f_j&=0&\quad&\text{for} \quad j = -M+1, \dots, M-1,\\
y_{j}&=Fj\eps&\quad&\text{for} \quad j = -M,\,M,
\end{aligned}
\end{equation}
where the atomistic force (per lattice spacing $\eps$) is given by
\begin{equation} \label{atomforce}
\begin{split}
F_{j}^a(\mathbf y) :=-\frac{1}{\eps}\frac{\partial \E^a(\yb)}{\partial y_j}
&=\frac{1}{\eps}\Bigg\{\left[\phi'\left( \frac{y_{j+1} - y_{j}}{\eps} \right)
+ \phi'\left(\frac{y_{j+2} - y_{j}}{\eps}\right)\right]\\
&\qquad\qquad-\left[\phi'\left( \frac{y_{j} - y_{j-1}}{\eps} \right)
+ \phi'\left(\frac{y_{j} - y_{j-2}}{\eps}\right)\right]\Bigg\}.
\end{split}
\end{equation}
In \eqref{atomforce} the undefined terms
$\phi'(\eps^{-1}(y_{-M+1}-y_{-M-1}))$ and
$\phi'(\eps^{-1}(y_{M+1}-y_{M-1}))$ are taken to be zero.

We let $u_j$ be a perturbation from the uniformly deformed state
$y_j^F = Fj\eps ,$ that is, we define
\begin{align*}
  & u_j = y_j - F j\eps  \quad \text{ for } j = -M,\dots,M.
\end{align*}
We linearize the atomistic equilibrium equations \eqref{full} about
the deformed state $\yb^F,$ resulting in
\begin{equation}
\label{atom}
\begin{aligned}
(L^{a} \ub^a)_j&=f_j&\quad&\text{for} \quad j = -M+1, \dots, M-1,\\
u^{a}_{j}&=0&\quad&\text{for} \quad j = -M,\,M,
\end{aligned}
\end{equation}
where $(L^a\vb)_j$, for a displacement $\vb \in \mathbb{R}^{2M+1}$, is
given by
\begin{equation*}
(L^{a} \vb)_j:=
\begin{cases}
\displaystyle
\phi''_F
	\left[\frac{-v_{j+1} + 2 v_{j} - v_{j-1}}{\eps^2} \right]
	+ \phi''_{2F}
	\left[\frac{-v_{j+2} + v_{j}}{\eps^2} \right], & j = -M+1, \\[10pt]
\displaystyle
\phi''_F
	\left[\frac{-v_{j+1} + 2 v_{j} - v_{j-1}}{\eps^2} \right]
	+ \phi''_{2F}
	\left[\frac{-v_{j+2} + 2 v_{j} - v_{j-2}}{\eps^2} \right], & j = -M+2,\dots,M-2, \\[10pt]
\displaystyle
\phi''_F
	\left[\frac{-v_{j+1} + 2 v_{j} - v_{j-1}}{\eps^2} \right]
	+ \phi''_{2F}
	\left[\frac{v_{j} - v_{j-2}}{\eps^2} \right], &j = M-1.
\end{cases}
\end{equation*}
Here and throughout we define $\phi''_{F} := \phi''(F)$ and
$\phi''_{2F} := \phi''(2F)$, where $\phi$ is the interatomic potential
in~\eqref{atomaa}.  We assume that $\phi''_F > 0$ and $\phi''_{2F} <
0,$ which holds for typical pair potentials such as the Lennard-Jones
potential under physically relevant deformations. We remark that, for
$\phi_F'' + 4 \phi_{2F}'' > 0$, the system \eqref{atom} has a unique
solution. This follows from \eqref{eq:defn_Ea} and from
Lemma~\ref{rdd_plus} (see also \cite{dobsonluskinopt,sharpstability}
for an analysis of the periodic case which is similar).

The local QC approximation uses the Cauchy-Born extrapolation rule to
approximate the nonlocal atomistic model by a local continuum
model~\cite{dobsonluskin07,minge05,tadmor_miller_qc_overview,tadmor_qc_first}.
In our context, this corresponds to approximating $y_{j}-y_{j-2}$ in
\eqref{atomaa} by $2(y_j-y_{j-1})$ and results in the local QC energy
\begin{equation}\label{lqcen}
\E^{lqc}(\yb)= \sum_{j = -M+1}^M \eps \left[\phi\left(\frac{y_j - y_{j-1}}{\eps}\right)
+ \phi\left( \frac{2(y_{j} - y_{j-1})}{\eps} \right) \right].
\end{equation}
Note that the above expression has one more next-nearest neighbor term
than~\eqref{atomaa}.  This is because the atoms at $j= -M+1, M-1$
do not feel the effect of the boundary in the local approximation.
The local quasicontinuum equilibrium equations are then given by
\begin{equation*}
\begin{aligned}
F_j^{lqc}(\yb) + f_j&=0&\quad&\text{for} \quad j = -M+1, \dots, M-1,\\
y_{j}&=Fj\eps& \quad&\text{for} \quad j = -M,\,M,
\end{aligned}
\end{equation*}
where the local quasicontinuum force (per lattice spacing $\eps$) is given by
\begin{equation} \label{lqcforce}
\begin{split}
F_{j}^{lqc}(\mathbf y):=-\frac{1}{\eps} \frac{\partial \E^a(\yb)}{\partial y_j}
&=\frac{1}{\eps}\Bigg\{\left[\phi'\left( \frac{y_{j+1} - y_{j}}{\eps} \right)
+ 2\phi'\left(\frac{2(y_{j+1} - y_{j})}{\eps}\right)\right]\\
&\qquad\qquad -\left[\phi'\left( \frac{y_{j} - y_{j-1}}{\eps} \right)
+ 2\phi'\left(\frac{2(y_{j} - y_{j-1})}{\eps}\right)\right]\Bigg\}.
\end{split}
\end{equation}
Linearizing the local quasicontinuum equilibrium equations
\eqref{lqcforce} about the deformed state $\yb^F$ results in
\begin{equation*}
\begin{aligned}
(L^{lqc} \ub^{lqc})_j&=f_j&\quad&\text{for} \quad j = -M+1, \dots, M-1,\\
u^{lqc}_{j}&=0&\quad&\text{for} \quad j = -M,\,M,
\end{aligned}
\end{equation*}
where $(L^{lqc}\vb)_j$, for a displacement $\vb \in
\mathbb{R}^{2M+1}$, is given by
\begin{equation*}
(L^{lqc} \vb)_j=
(\phi''_F + 4 \phi''_{2F})
	\left[\frac{-v_{j+1} + 2 v_{j} - v_{j-1}}{\eps^2} \right], \quad j=-M+1,\dots, M-1.
\end{equation*}

The increased efficiency of the local QC approximation is obtained
when its equilibrium equations~\eqref{lqcforce} are coarsened by
reducing the degrees of freedom using piecewise linear interpolation
between a subset of the
atoms~\cite{dobsonluskin07,tadmor_miller_qc_overview}.  For the sake
of simplicity of exposition, we do not treat coarsening in this paper.

In order to combine the accuracy of the atomistic model with the
efficiency of the local QC approximation, the force-based QC method
decomposes the reference lattice into an {\it atomistic region}
$\mathcal{A}$ and a {\it continuum region} $\mathcal{C}$, and assigns
forces to atoms according to the region they are located in. Since the
local QC energy~\eqref{lqcen} approximates $y_{j}-y_{j-2}$ in
\eqref{atoma} by $2(y_j-y_{j-1}),$ it is clear that the atomistic
model should be retained wherever the strains are varying rapidly. The
QCF operator is then given by~\cite{dobsonluskin07,dobsonluskin08}
\begin{equation}\label{qcfdefF}
F_{j}^{qcf}(\mathbf y)=\begin{cases}
F_{j}^{a}(\mathbf y)& \text{if $j\in \mathcal{A}$},\\
F_{j}^{lqc}(\mathbf y)& \text{if $j\in \mathcal{C}$},
\end{cases}
\end{equation}
and the QCF equilibrium equations by
\begin{equation*}
\begin{aligned}
F_j^{qcf}(\yb) + f_j&=0&\quad&\text{for} \quad j = -M+1, \dots, M-1,\\
y_{j}&=Fj\eps& \quad&\text{for} \quad j = -M,\,M.
\end{aligned}
\end{equation*}
The force-based quasicontinuum approximation gets its name from the
assignment of forces at the atoms in~\eqref{qcfdefF}.  Most other
quasicontinuum approximations build a total energy by summing energy
contributions from each region and compute forces on the atoms by
differentiating the energy.  However, $F^{qcf}$ is a non-conservative
force field and cannot be derived from an
energy~\cite{dobsonluskin07}.

\subsection{Artificial boundary conditions for the computational domain}
For large atomistic systems, it is necessary to reduce the
computational domain, even when using a coarse-graining method such as
the QC approximation.  The reduction of the computational domain
requires the use of artificial boundary conditions to approximate the
effect of the far field.  The artificial boundary condition most
commonly used in the QC method (and in computations using other
atomistic to continuum approximations) sets the displacement to zero
at the boundary of the computational domain, such as at the lateral
boundary of the crystal in the nanoindentation problem reported in
~\cite{Knap:2001a}.  More accurate artificial boundary conditions such
as given and analyzed in~\cite{lee:1749} do not seem to have yet been
used in quasicontinuum computations.

We chose to imitate the approach commonly used in the QC method, by
choosing $N \ll M$ and $0 < K < N-1$, letting $\{-N,\dots,N\}$ be the
computational domain, $\mathcal{A} = \{-K,\dots,K\}$ the atomistic
region, and $\mathcal{C} = \{-N+1, \dots, N-1\} \setminus \mathcal{A}$
the continuum region, and defining
\begin{equation}
\label{fullqcfc}
\begin{aligned}
F_j^{qcf}(\yb) + f_j&=0&\quad&\text{for} \quad j = -N+1, \dots, N-1,\\
y_{j}&=Fj\eps& \quad&\text{for} \quad j = -N,\,N,
\end{aligned}\end{equation}
to be the QCF approximation on the computational domain.  In this
paper, we analyze the linearization of ~\eqref{fullqcfc},
\begin{equation}
\label{qcf}
\begin{aligned}
(L^{qcf} \ub^{qcf})_j&=f_j&\quad&\text{for} \quad j = -N+1, \dots, N-1,\\
u^{qcf}_{j}&=u^a_j&\quad&\text{for} \quad j = -N,\,N,
\end{aligned}
\end{equation}
where we have taken $u^{qcf}_{-N}=u^a_{-N}$ and $u^{qcf}_N=u^a_{N}$ so
that we may ignore the error induced by the artificial boundary
condition and exclusively focus on the error of the QC
approximation. Note that, since atoms near the artificial boundary
belong to $\mathcal{C}$, only one boundary condition is required at
each end.

Setting $\eps = 1/N$ throughout, we scale the problem
~\cite{BlancLeBrisLions:2007} so that the size of the computational
domain is of order O$(1)$.

\subsection{Notation}
We use $D : \mathbb{R}^{2N+1} \to \mathbb{R}^{2N}$ to denote the
backward difference operator, defined by
\begin{equation*}
  (D\vb)_j = D v_j = \frac{v_j - v_{j-1}}{\eps} \quad \text{ for }
  j = -N+1, \dots, N.
\end{equation*}
We will frequently employ the weighted $\ell^p$-norms,
\begin{equation*}
\begin{split}
\lpnorm{\vb}{p} &:= \Bigg(\eps \sum^{N}_{j=-N}  |v_j|^p\Bigg)^{1/p},
   \qquad 1\le p<\infty,\\
\lpnorm{\vb}{\infty} &:= \max_{-N\leq j\leq N}  |v_j|,
\end{split}
\end{equation*}
and the weighted inner product
\begin{displaymath}
 \langle {\bf v}, {\bf w} \rangle = \sum_{j = -N}^N \eps v_j w_j.
\end{displaymath}
The definition of the difference operator $D$, of the norms
$\|\vb\|_{\ell^p_\eps}$ and of the inner product $\langle \vb, \wb
\rangle$ is extended, in an obvious way, for vectors $\vb, \wb \in
\mathbb{R}^K$, where $K \in \mathbb{N}$ is arbitrary. For example, if
$\vb \in \mathbb{R}^{2M+1}$ then $D\vb \in \mathbb{R}^{2M}$. Moreover,
in view of this convention, the higher order difference operators
$D^2$, etc., can be defined by successive application of $D$; for
example, $D^2 v_j = \eps^{-2}(v_j - 2 v_{j-1} + v_{j-2})$.

The subspace of $\mathbb{R}^{2N+1}$ with homogeneous boundary
conditions is denoted
\begin{equation*}
\Ds  = \big\{ \vb \in \mathbb{R}^{2N+1} : v_{-N} = v_N = 0 \big\}.
\end{equation*}
For future reference we note that the following Poincar\'{e}
inequality holds \cite[Lemma A.3]{ortnersuli}:
\begin{equation}
  \label{eq:poincare}
  \| \vb \|_{\ell^\infty_\eps} \leq {\textstyle \frac{1}{2}}
  \| D\vb \|_{\ell^1_\eps} \qquad \text{for all } \vb \in \Ds.
\end{equation}

Furthermore, we note that the linear operator $L^{qcf}$ which has been
defined above as a mapping from $\Real^{2N+1}$ to $\Real^{2N-1}$ will
be considered below to be a mapping from $\Real^{2N+1}$ to $\Ds$ by
the extension
\[
(L^{qcf}\vb)_{-N}=(L^{qcf}\vb)_{N}=0 \qquad \text{for }
\vb\in\Real^{2N+1}.
\]
With this in mind, $\langle L^{qcf} \vb, \wb \rangle$ is well-defined
for all $\vb, \wb \in \mathbb{R}^{2N+1}$.

\subsection{Pointwise consistency of the force-based QC approximation}
The remarkable simplicity of the formulation of the force-based QC
approximation is mirrored by its equally simple consistency
analysis. Let $\ub^a$ be the solution to \eqref{atom} (assuming
$\phi_F'' + 4 \phi_{2F}'' > 0$, this system is well-posed), then the
truncation error ${\bf t} \in \mathbb{R}^{2N+1}$ is defined by $t_{-N}
= t_N = 0$ and
\begin{displaymath}
  t_j = (L^{qcf} \ub^a - \fb)_j = (L^{qcf}\ub^a - L^{a} \ub^a)_j
  \qquad \text{for} \quad j = -N+1, \dots, N-1,
\end{displaymath}
where $L^{qcf}\ub^a$ is understood by restricting $\ub^a$ to the
computational domain.  Since $(L^{qcf}\ub^a)_j = (L^a\ub^a)_j$
trivially holds for $j \in \mathcal{A}$ we have $t_j = 0$ for $j \in
\mathcal{A}$. For $j \in \mathcal{C}$, on the other hand, we have
\begin{align*}
  t_j = (L^{qcf}\ub^a - L^a \ub^a)_j =\,& \phi_{2F}'' \Bigg[ 4 \frac{-u^a_{j+1} + 2 u^a_j - u^a_{j-1}}{\eps^2} - \frac{-u^a_{j+2} + 2 u^a_j - u^a_{j-2}}{\eps^2} \Bigg] \\
  =\,& \eps^2 \phi_{2F}'' \Bigg[ \frac{u^a_{j+2} - 4 u^a_{j+1} + 6 u^a_j - 4 u^a_{j-1} + u^a_{j-2}}{\eps^4} \Bigg] = \eps^2 \phi_{2F}'' (\bar D^{4} \ub^a)_j,
\end{align*}
where $(\bar D^4 \vb)_j=(D^4\vb)_{j+2}$ is a fourth-order centered
finite difference operator.  Note also that $\ub^a$ is defined outside
the computational domain.  Thus, for $p \in [1, \infty]$, we obtain an
exact truncation error estimate,
\begin{equation}
  \label{eq:qcf_const_1}
  \big\| {\bf t} \big\|_{\ell^p_\eps}
  = \eps^2 |\phi_{2F}''| \| \bar D^4 \ub^a \|_{\ell^p_\eps(\mathcal{C})},
\end{equation}
where the label $\mathcal{C}$ indicates that the summation (or
maximum) is only taken over atoms in the continuum region.

We have presented (\ref{eq:qcf_const_1}) as a simple argument for the
high accuracy of the QCF method, however, in the error analysis in
Section \ref{sec:converge} we will use a slightly sharper
negative-norm estimate. We also note that it follows from the interior
regularity of elliptic finite difference operators~\cite{thomee68}
that $\| \bar D^4 \ub^a \|_{\ell^p_\eps(\mathcal{C})}$ is bounded in
the continuum limit $\eps\to 0$ if $\fb$ is the restriction of a
smooth function in a neighborhood of the continuum region
$\mathcal{C}$ to the lattice points since the continuum region
$\mathcal{C}$ is far from the boundary of the atomistic problem if
$N\ll M.$

To estimate the error between the atomistic and QCF solution, we write
\begin{displaymath}
  L^{qcf}(\ub^a - \ub^{qcf}) = {\bf t}
  = O(\eps^2 |\phi_{2F}''|).
\end{displaymath}
Hence, a uniform stability result for the operator $L^{qcf}$ in an
appropriate norm would lead to an optimal error estimate.  As we have
already remarked in the introduction and will make precise in Theorem
\ref{th:nocoerc}, $L^{qcf}$ is typically {\em not} coercive and we
must therefore prove an inf-sup condition instead. To this end, we
will factor the $L^{qcf}$ operator into divergence form, $L^{qcf} =
D^T E^{qcf} D,$ where $D$ is the discrete difference operator defined
above.  We will give conditions under which $E^{qcf}$ is row
diagonally-dominant and which will lead to an inf-sup condition for
$L^{qcf}.$ Interestingly, however, this approach only leads to uniform
stability bounds if the $\ell^\infty_\eps$-$\ell^1_\eps$ duality
pairing is used, while the inf-sup constants for the
$\ell^p_\eps$-$\ell^q_\eps$ ($1/p+1/q = 1, \, 1 \leq p < \infty$)
pairings are not uniform in $N$ (cf. Section \ref{sec:norms}).

\section{Divergence Form of the QCF Operator}
\label{sec:conjugate}

We will analyze the QCF equilibrium equations~\eqref{qcf} by putting
them into a ``weak form:'' {\it find $\ub^{qcf}\in\Real^{2N+1}$ such
  that }
\begin{equation*}
\begin{split}
\langle E^{qcf} D\ub^{qcf}, D\wb \rangle =\,& \langle \fb, \wb \rangle \qquad \forall \, \wb \in \Ds, \\
u^{qcf}_j =\,& u^a_j \qquad  \text{for } j = -N,\,N,
\end{split}
\end{equation*}
where the linear operator $E^{qcf}:\Real^{2N}\to\Real^{2N} $ is chosen
so that $\langle E^{qcf} D\vb, D\wb \rangle=\langle L^{qcf} \vb, \wb
\rangle$ for all $\vb \in \Real^{2N+1}$ and $\wb \in \Ds$.  We call $E^{qcf}$ the
conjugate operator.  This operator was previously derived for a
Neumann problem in~\cite{dobsonluskin07} and for a problem with mixed
boundary conditions in~\cite{dobsonluskin08}.

To motivate the idea we briefly review the conjugate operator for the
full atomistic model before deriving $E^{qcf}$. The atomistic energy
(and similarly all QC energies) can be written as functions of the
strain $D\ub,$ $\widehat{\E}^{a}(D\ub):={\E}^{a}(\yb),$ and its
conjugate operator is defined by~\cite{dobsonluskin07,dobsonluskin08}
  \begin{equation*}
    \left(E^{a} (\rb)\right)_j := \frac 1\eps \pd{\widehat{\E}^{a}}{r_j}(\rb).
  \end{equation*}
  Thus, $\left(E^{a} (\rb)\right)_j$ is the negative of the force
  conjugate to the strain $(D \ub)_j.$ It follows from the chain rule
  that
  \begin{equation*}
    \begin{split}
      F_{j}^{a}(\mathbf y)
      &:=-\frac{1}{\eps}\frac{\partial \E^{a}(\yb)}{\partial y_j}
      =\frac{1}{\eps}\left[\pd{\widehat{\E}^{a}}{r_{j+1}}(\rb)
        -\pd{\widehat{\E}^{a}}{r_j}(\rb)\right] \\
      &=\left[  \left( E^{a} (\rb)\right)_{j+1}
        -\left(E^{a} (\rb)\right)_j\right].
    \end{split}
  \end{equation*}
  Applying this calculation to the linearized operator $L^a,$ one can
  easily verify that $\langle L^{a} \vb, \wb \rangle = \langle E^a D
  \vb, D\wb \rangle$, for all $\vb \in \mathbb{R}^{2M+1}$ and $\wb \in
  \mathbb{R}^{2M+1}$ with $w_{-M} = w_{M} = 0$, where
\begin{equation}
  \label{eq:defn_Ea}
  E^a = \phi_F'' I  + \phi_{2F}'' \left[
    \begin{matrix}
      1 &      1 &        &        &    \\
      1 &      2 & 1      &        &    \\
        & \ddots & \ddots & \ddots &    \\
        &        &      1 &      2 & 1  \\
        &        &        &      1 & 1
    \end{matrix}
  \right].
\end{equation}
From this representation we obtain immediately that, for $\vb \in
\mathbb{R}^{2M+1}$ and for $\wb \in \Ds$, extended by zero outside the
computational domain, the operator $L^a$ can be written in the
``weak'' form
\begin{equation}
  \label{eq:La_weak}
  \langle L^a \vb, \wb \rangle = \sum_{j = -N+1}^N \big[
  \phi_F''Dv_j + \phi_{2F}'' \big( Dv_{j+1} + 2 Dv_j
  + Dv_{j-1}\big) \big] Dw_j.
\end{equation}
This formula will be used in Section \ref{sec:converge} to derive a
negative-norm truncation error estimate.

To find a representation for the QCF operator $L^{qcf}$ in terms of a
conjugate operator we cannot simply carry out the same computation as
above, even in the linearized case, since it is not related to any
energy functional. Instead, we will first derive a ``weak'' form for
$L^{qcf}$ from which it will be fairly straightforward to construct
the conjugate operator. We begin by writing $L^{qcf}$ in the form
$L^{qcf} = \phi_{F}'' L_1 + \phi_{2F}'' L_2$, where
\begin{align*}
  (L_1 \vb)_j =~&
  \eps^{-2} \big( -v_{j+1} + 2 v_j - v_{j-1} \big), \qquad j = -N+1,
  \dots, N-1,  \qquad \text{and} \\
  (L_2 \vb)_j =~& \begin{cases}
    4\eps^{-2} \big( - v_{j+1} + 2 v_j - v_{j-1} \big), &
    j \in \mathcal{C}, \\
    \eps^{-2} \big( -v_{j+2} + 2 v_j - v_{j-2} \big), & j \in
    \mathcal{A},
  \end{cases}
\end{align*}
and deriving ``weak'' representations of the operators $L_1$ and
$L_2$.

\begin{lemma}
  \label{th:ibp_lemma}
  For all $\vb \in \Real^{2N+1}$ and $\wb \in \Ds$ the nearest neighbor and next-nearest neighbor
  interaction operators can be written in the form
  \begin{align*}
    \langle L_1 \vb, \wb \rangle =~& \sum_{j = -N+1}^N \eps D v_j D w_j,
    \qquad \text{and} \\
    \langle L_2 \vb, \wb \rangle =~& \langle L_2^{reg} \vb, \wb \rangle
    + \eps^2 (D^3 v_{-K+1}) w_{-K} - \eps^2(D^3 v_{K+2}) w_K,
  \end{align*}
  where $D^3$ is the third-order  backward finite difference operator, $D^3 v_j = \eps^{-2}(Dv_j - 2 Dv_{j-1} + Dv_{j-2}),$ and where
  $L_2^{reg}$ denotes the ``regular'' component of $L_2$,
  \begin{displaymath}
    \langle L_2^{reg} \vb, \wb \rangle = \sum_{j = -N+1}^{-K}
    \eps 4 Dv_j Dw_j
    + \sum_{j = -K+1}^K \eps (Dv_{j-1} + 2 Dv_j + Dv_{j+1}) Dw_j
    + \sum_{j = K+1}^{N} \eps 4 Dv_j Dw_j.
  \end{displaymath}
\end{lemma}
\begin{proof}
  We only prove the representation for $L_2$.  To simplify the
  notation, we will perform all manipulations only in the {\em right}
  half of the domain and indicate the remaining terms by dots, for
  example,
  \begin{displaymath}
    \langle L_2 {\bf v}, {\bf v} \rangle
    = \dots +  \sum_{j = 0}^K \eps (L_2 {\bf v})_j w_j
    + \sum_{j = K+1}^{N-1} \eps (L_2 {\bf v})_j w_j.
  \end{displaymath}

  The proof simply requires careful summation by parts, performed
  separately in the continuum and atomistic region. In the right half
  of the atomistic region, summation by parts yields
  \begin{align*}
    \sum_{j = 0}^{K} \eps (L_2 {\bf v})_j w_j
    =~& - \sum_{j = 0}^{K} \Big[ \Big(\frac{v_{j+2} - v_j}{\eps}\Big)
    - \Big(\frac{v_{j} - v_{j-2}}{\eps}\Big) \Big] w_j \\
    =~& - \sum_{j = {2}}^{K+2} \Big(\frac{v_{j} - v_{j-2}}{\eps}\Big) w_{j-2}
    + \sum_{j = 0}^{K} \Big(\frac{v_{j} - v_{j-2}}{\eps}\Big) w_{j} \\
     =~& \dots + \sum_{j = 2}^{K} \eps \big(D v_j + Dv_{j-1} \big)
    \big( Dw_j + Dw_{j-1} \big) \\
    & \hspace{5mm} - \big[ (Dv_{K+1} + Dv_K) w_{K-1}
    + (Dv_{K+2} + Dv_{K+1}) w_{K} \big] \\
    =~& \dots + \sum_{j = 1}^{K} \eps \big(D v_{j+1} + 2 D v_j
    + Dv_{j-1} \big) Dw_j  \\
    & \hspace{5mm} - (Dv_{K+2} + 2 Dv_{K+1} + Dv_K) w_{K}.
  \end{align*}
  Here, we also used the dots to indicate additional terms which would
  have canceled had we performed the calculation over the entire
  domain.  A similar computation in the continuum region gives
  \begin{displaymath}
    \sum_{j = K+1}^{N-1} \eps (L_2 {\bf v})_j v_j = \sum_{j = K+1}^N
    \eps 4 D v_j Dw_j + 4 Dv_{K+1} w_{K} .
  \end{displaymath}
  Considering the symmetry of the problem, or by performing the same
  calculation in the left half of the domain, we obtain the stated
  result.
\end{proof}

In order to find the conjugate operator, we only need to write $w_K$
and $w_{-K}$ in terms of the strains $Dw_j$. This is achieved by
connecting these displacements to the boundary, for example, we can
use the identities
\begin{displaymath}
  w_K = - \sum_{j = K+1}^N \eps Dw_j \quad \text{and} \quad
  w_{-K} = \sum_{j = -N+1}^{-K} \eps Dw_j.
\end{displaymath}
Note, however, that there is no unique way of achieving this. Our
choice above simply minimizes the number of non-zero entries for
$E^{qcf}$ in each row, a fact that will become important later
on. Thus, we obtain
\begin{equation}
  E^{qcf} = \phi''_F I + \phi''_{2F} \left[
  \begin{array}{cccccccccccccccc}
4&&&1&\hbox{-}2&1&&&&&&&\\[-2.2mm]
&\ddots&&\vdots&\vdots&\vdots&&&&&&&\\[-0.5mm]
&&4&1&\hbox{-}2&1&&&&&&&\\
&&&5&\hbox{-}2&1&&&&&&&\\
&&&1&2&1&&&&&&&\\
&&&&1&2&1&&&&&&\\[-2.2mm]
&&&&&\ddots&\ddots&\ddots&&&&&\\[-0.5mm]
&&&&&&1&2&1&&&&\\
&&&&&&&1&2&1&&&\\
&&&&&&&1&\hbox{-}2&5&&&\\
&&&&&&&1&\hbox{-}2&1&4&&\\[-2.2mm]
&&&&&&&\vdots&\vdots&\vdots&&\ddots&\\[-0.5mm]
&&&&&&&1&\hbox{-}2&1&&&4\\
\end{array}
    \right].
  \label{eqcfd}
\end{equation}
With this choice we indeed obtain the identity $\langle E^{qcf} D\vb,
D\wb \rangle=\langle L^{qcf} \vb, \wb \rangle$ for all $\vb \in \Real^{2N+1}$ and $\wb \in
\Ds$.

\section{Lack of Coercivity}
\label{sec:nocoerc}

Local minimizers of the atomistic energy are characterized,
essentially, by the fact that the Hessian is positive definite.  It
can be shown that the coercivity of the operator $L^a$ in appropriate
norms is independent of problem size $N$, provided that $\phi_F'' + 4
\phi_{2F}'' > 0$ (see \cite{dobsonluskinopt,sharpstability} for a
periodic problem; the proof for the Dirichlet boundary value problem
is very similar).  It can moreover be shown that the {\em
  energy-based} quasi-nonlocal approximation always inherits
coercivity of the atomistic operator~\cite{sharpstability}.

In Section~\ref{sec:stability}, we will give conditions on $\phi''_F$
and $\phi''_{2F}$ under which the force-based QC operator $L^{qcf}$
inherits stability (in a more general sense) uniformly in $N,$ even
though it does {\em not} inherit coercivity.  In fact, as we show in
the following theorem, the $L^{qcf}$ operator is not coercive whenever
$N$ is sufficiently large.  That is, not only is it lacking uniform
coercivity, but it is not even positive definite.  This result clearly
demonstrates why we need to work with a technically more involved
inf-sup stability condition in our error analysis in later sections.

\begin{theorem}
  \label{th:nocoerc}
  Suppose that $\phi_F'' > 0$ and $\phi_{2F}'' \in \mathbb{R}
  \setminus \{0\};$ then, for sufficiently large $N,$ the operator
  $L^{qcf}$ is {\em not} coercive. More precisely, there exist $N_0
  \in \mathbb{N},\, C_1 \geq C_2 > 0$ such that, for all $N \geq N_0$ and
  $2 \leq K \leq N/2$,
  \begin{displaymath}
    - C_1 N^{1/2} \leq  \inf_{\substack{\vb \in \Ds \\ \|D\vb\|_{\ell^2_\eps} = 1}}
    \langle L^{qcf} {\bf v}, {\bf v} \rangle
    \leq - C_2 N^{1/2}.
  \end{displaymath}
\end{theorem}
\begin{proof}
  As in Section \ref{sec:conjugate} we write $L^{qcf} = \phi_{F}'' L_1
  + \phi_{2F}'' L_2$. Since $\langle L_1 {\bf v}, {\bf v}
  \rangle=\|D\vb\|_{\ell^2_\eps}^2,$ we need to concentrate on the
  next-nearest neighbor interaction operator $L_2$.
  If we can show
  that $L_2$ is neither bounded above nor below, uniformly in $N$, and
  with the stated asymptotic behavior, then the upper bound follows.
  The lower bound follows from the fact that $L_1$ is bounded while
  $|\langle L_2 \vb, \vb \rangle| \leq C_1 N^{1/2}
  \|D\vb\|_{\ell^2_\eps}^2$. Both of these facts are established in
  the following lemma.
\end{proof}

\begin{lemma}
  \label{th:cocoerc_lemma}
  Under the conditions of Theorem \ref{th:nocoerc}, there exist
  positive constants $c_1,\, c_2$, independent of $N$, and lattice
  functions $\vb^+,\, \vb^- \in \Ds$ such that
  $\lpnorm{D\vb^+}{2}=\lpnorm{D\vb^-}{2}=1,$
  \begin{displaymath}
    \langle L_2 {\bf v}^+, {\bf v}^+ \rangle \geq c_1 (N^{1/2} - c_2),
    \quad \text{and} \quad
    \langle L_2 {\bf v}^-, {\bf v}^- \rangle \leq -c_1 (N^{1/2} - c_2).
  \end{displaymath}
  Moreover, these bounds are asymptotically optimal in that there
  exists a constant $c_3 > 0$ such that
  \begin{displaymath}
    \big| \langle L_2 {\bf v}, {\bf v} \rangle \big| \leq c_3 N^{1/2}
    \qquad \text{for all } \vb \in \Ds \text{ with }
    \|D\vb\|_{\ell^2_\eps} = 1.
  \end{displaymath}
\end{lemma}
\begin{proof}
  We write $\langle L_2 {\bf v}, {\bf v} \rangle$ by setting $\wb =
  \vb$ in Lemma \ref{th:ibp_lemma}. The crucial observation is that
  the term $v_K (Dv_{K+2} - 2 Dv_{K+1} + Dv_{K})$ cannot be
  expressed as a quadratic
  form of strains supported at the interface, while all other
  terms are bounded in terms of (a constant multiple of) $\| D\vb
  \|_{\ell^2_\eps}^2$. More precisely, we recall that
  \begin{displaymath}
    \langle L_2 \vb, \vb \rangle = \langle L_2^{reg} \vb, \vb \rangle
    - v_K (Dv_{K+2} - 2 Dv_{K+1} + Dv_{K})
    + v_{-K}(Dv_{-K-1} - 2Dv_{-K} + Dv_{-K+1}),
  \end{displaymath}
  where $|\langle L_2^{reg} \vb, \vb \rangle| \leq c_2 \| D\vb
  \|_{\ell^2_\eps}^2$.  Next, we construct the
  functions $\vb^\pm$ by choosing $v_K = 1$ and so that the third
  difference in the bracket is of order $N^{1/2}$.

  To this end, we set ${\bf v} = \bar {\bf v} + \eps^{1/2} {\bm
    \delta}_{K+1} = \bar {\bf v} + N^{-1/2} {\bm
    \delta}_{K+1}$, where
  \begin{displaymath}
    \bar v_j = \left\{
      \begin{array}{rl}
        (N+j)/(N-K-2), & \quad j = -N, \dots, -K-2 \\
        1, & \quad j = -K-2, \dots, K+2, \\
        (N-j) / (N-K-2), & \quad j = K+2, \dots, N,
      \end{array}
    \right.
  \end{displaymath}
  (that is, $\bar v_j = 1$ in the atomistic region and the interface,
  and interpolates linearly between 1 and 0 in the continuum region)
  and where $\delta_{K+1, j} = 0$ if $j \neq K+1$ and $\delta_{K+1,\, K+1} =
  1$. In that case, $\|D {\bf v} \|_{\ell^2_\eps}$ is clearly
  uniformly bounded, and we obtain
  \begin{align*}
    \langle L_2 {\bf v}, {\bf v} \rangle =~&
    \langle L_2^{reg} \vb, \vb \rangle + 3 N^{1/2}.
  \end{align*}
  Note that no terms at the left interface occur since $\vb$ is a
  constant there.  Upon appropriately rescaling by ${\bf v}^+ = {\bf
    v}/\|D{\bf v}\|_{\ell^2_\eps}$ so that $\|D\vb^+\|_{\ell^2_\eps} =
  1$, we obtain
  \begin{displaymath}
     \langle L_2 {\bf v}^+, {\bf v}^+ \rangle \geq -c_2 + c_1 N^{1/2}.
  \end{displaymath}
  Setting ${\bf v}^- = c(\bar {\bf v} - \eps^{1/2} {\bm
    \delta}_{K+1})$ gives the opposite bound.

  To prove the final statement, namely that these bounds are
  asymptotically sharp, we note that all terms of the type $v_K Dv_j$
  are of order $N^{1/2}$,
  \begin{displaymath}
    \big| v_{K} Dv_j \big|
    = \eps^{-1/2} |v_{K}| \eps^{1/2} |Dv_{j}|
    \leq \eps^{-1/2} \|{\bf v}\|_{\ell^\infty_\eps} \|D {\bf v}\|_{\ell^2_\eps}
    \leq (2/\eps)^{1/2}  \|D {\bf v} \|_{\ell^2_\eps}^2,
  \end{displaymath}
  where we used (\ref{eq:poincare}) and a weighted Cauchy--Schwartz
  inequality to bound $\|{\bf v}\|_{\ell^\infty_\eps} \leq \sqrt{2} \|
  D{\bf v} \|_{\ell^2_\eps}$.
\end{proof}

\begin{remark}
  The proof of Lemma \ref{th:cocoerc_lemma} reveals that $N$ needs to
  be of the order $(1 + |\phi_{F}''/\phi_{2F}''|)^{2}$ before a loss
  of coercivity can occur. Although it may seem that this is typically
  a fairly large number, $(1 + |\phi_{F}''/\phi_{2F}''|)^{2}$ is not
  so large for strains $F$ near the edge of a stability region (such
  as near the critical strain at which the atomistic system
  ``fractures''~\cite{blan05}), or more generally whenever the
  next-nearest neighbor interaction is not significantly dominated by
  the nearest neighbor interaction.
\end{remark}
\section{Stability of the Force-Based Quasicontinuum Solution}
\label{sec:stability}

We first recall a classical characterization of the norm of the
inverse of an operator that we will use to prove the stability of the
solution to the force-based quasicontinuum approximation. The proof is
included for the sake of completeness.
\begin{lemma}[Inf-Sup Condition]
  \label{lem:infsup}
  Let $W$ and $V$ be finite dimensional normed linear spaces
  satisfying $\dim W = \dim V,$ and let $L$ be a bounded linear
  operator from $V$ to $W'$ where $W'$ is the dual of $W.$ Suppose
  that
  \begin{equation}
    \inf_{\substack{\vb \in V \\ \|\vb\|_V = 1}} \
    \sup_{\substack{\wb \in W \\ \|\wb\|_W = 1}}
    \langle L \vb, \wb \rangle = \gamma > 0.
    \label{eq:infsup}
  \end{equation}
  Then $L$ is invertible and the solution $\ub \in V$ to
  $L\ub = \fb$ satisfies the stability bound
  \begin{equation*}
    \|\ub\|_V \leq \frac{1}{\gamma} \|\fb\|_{W'} \quad \text{ where }\quad
    \|\fb\|_{W'} := \sup_{\substack{\wb \in W \\ \|\wb\|_W = 1}}
    \langle \fb, \wb \rangle.
  \end{equation*}
\end{lemma}
\begin{proof}
  The inf-sup condition~\eqref{eq:infsup} implies that the nullspace
  of $L$ must be trivial.  Since a finite-dimensional linear operator
  between two spaces of the same dimension is invertible if and only
  if it is non-singular, we conclude that there is a unique solution
  $\ub \in V$ to $L\ub = \fb$ for every $\fb \in W'.$

  If $\|\ub\|_V = 0,$ then the stability bound is trivial.  Otherwise,
  we have
  \begin{equation*}
    \begin{split}
      \| \fb \|_{W'} = \sup_{\substack{\wb \in W \\ \|\wb\|_W = 1}}
      \langle L \ub, \wb \rangle
      = \|\ub\|_V \sup_{\substack{\wb \in W \\ \|\wb\|_W = 1}}
      \left\langle L \left(\frac{\ub}{\|\ub\|_V}\right), \wb \right\rangle
      \geq \gamma \|\ub\|_V. \qedhere
    \end{split}
  \end{equation*}
\end{proof}

Next, we note that the range of the backward difference operator $D$
is
\begin{equation*}
\R(D) = \mvz :=\Bigg\{ \xib \in \Real^{2N} : \sum_{j=-N+1}^N \xi_j = 0 \Bigg\},
\end{equation*}
 and therefore
\begin{equation*}
\begin{split}
\inf_{\substack{\vb \in \Ds \\ \lpnorm{D \vb}{\infty} = 1}} \
\sup_{\substack{\wb \in \Ds \\ \lpnorm{D \wb}{1} = 1}}
    \langle  L^{qcf} \vb,\,\wb \rangle
&=
\inf_{\substack{\vb \in \Ds \\ \lpnorm{D \vb}{\infty} = 1}} \
\sup_{\substack{\wb \in \Ds \\ \lpnorm{D \wb}{1} = 1}}
    \langle   E^{qcf} D \vb,\,D \wb \rangle  \\
&=
\inf_{\substack{\xib \in \mvz \\ \lpnorm{\xib}{\infty} = 1}}
\sup_{\substack{\etab \in \mvz \\ \lpnorm{\etab}{1} = 1}}
    \langle  E^{qcf} \xib ,\,\etab\rangle.
\end{split}
\end{equation*}

The following lemma gives a bound on such an inf-sup constant, for a
general matrix $A$. This result and its proof were inspired by
\cite[Sec. 3.1]{ortnersuli}.

\begin{lemma}
  \label{rdd_plus}
  Suppose that $A \in \Real^{2N \times 2N}$ satisfies
  \begin{equation*}
    \min_{i} \Big( A_{ii} + \sum_{j \neq i} A_{ij}^- \Big)
    - \max_{i} \sum_{j \neq i} A_{ij}^+  =: \gamma > 0,
  \end{equation*}
  where $A_{ij}^- = \min(0, A_{ij})$ and $A_{ij}^+ = \max(0, A_{ij})$,
  then
\begin{equation*}
\inf_{\substack{\xib \in \mvz \\ \lpnorm{\xib}{\infty} = 1}}
\sup_{\substack{\etab \in \mvz \\ \lpnorm{\etab}{1} = 1}}
    \langle A \xib,\,\etab \rangle \geq \gamma/2.
\end{equation*}
\end{lemma}
\begin{proof}
  Let $\xib \in \mathbb{R}^{2N}_* \setminus \{0\}$ and choose $p, q
  \in \{-N+1, \dots, N\}$ such that $\xi_p = \max_j \xi_j$ and $\xi_q
  = \min_j \xi_j$. Since $\sum_{j = -N+1}^N \xi_j = 0$, it follows
  that $\xi_p > 0$ and $\xi_q < 0$. Moreover, let $P = \{ j : \xi_j
  \geq 0\}$ and $Q = \{ j : \xi_j < 0\}$.  If we define $\etab \in
  \mathbb{R}^{2N}_*$ by
\begin{equation*}
\eta_i = \begin{cases}
\frac{1}{2 \eps}, & \quad i = p, \\
-\frac{1}{2 \eps}, & \quad i = q, \\
0, & \quad \text{otherwise,}
\end{cases}
\end{equation*}
then
\begin{align*}
  2\langle A \xib, \etab \rangle =~& \Big\{\sum_{j} A_{pj} \xi_j\Big\}
  - \Big\{\sum_{j} A_{qj} \xi_j\Big\} \\
  \geq~& \Big\{A_{pp} \xi_p + \sum_{j \in Q} A_{pj}^+ \xi_j +\sum_{j \in P \setminus \{p\}} A_{pj}^- \xi_j\Big\}
  - \Big\{A_{qq} \xi_q + \sum_{j \in P} A_{qj}^+ \xi_j + \sum_{j \in Q \setminus \{q\}} A_{qj}^- \xi_j \Big\} \\
  \geq~& \Big\{A_{pp} \xi_p + \sum_{j \in Q} A_{pj}^+ \xi_q +\sum_{j \in P \setminus \{p\}} A_{pj}^- \xi_p\Big\}
  - \Big\{A_{qq} \xi_q + \sum_{j \in P} A_{qj}^+ \xi_p + \sum_{j \in Q \setminus \{q\}} A_{qj}^- \xi_q \Big\} \\
  =~& \Big[A_{pp} - \sum_{j \in P \setminus \{p\}} |A_{pj}^-| - \sum_{j \in P} |A_{qj}^+|\Big] |\xi_p|
  + \Big[A_{qq} - \sum_{j \in Q \setminus \{q\}} |A_{qj}^-| - \sum_{j \in Q} |A_{pj}^+| \Big] |\xi_q| \\
  \geq~& \gamma (|\xi_p| + |\xi_q|). \qedhere
\end{align*}
\end{proof}

\medskip \noindent From Lemma~\ref{rdd_plus} and from (\ref{eqcfd}),
we can now deduce that
\begin{align}
  \notag
  \inf_{\substack{\vb\in\Ds \\ \lpnorm{D \vb}{\infty} = 1}} \
  \sup_{\substack{\vb\in\Ds \\ \lpnorm{D \wb}{1} = 1}}
  \langle  L^{qcf} \vb,\,\wb \rangle &\geq
  \frac{1}{2} \Bigg[
  \min_i \Big( (E^{qcf})_{ii} + \sum_{j \neq i} (E^{qcf})_{ij}^- \Big)
  - \max_i \sum_{j \neq i} (E^{qcf})_{ij}^+ \Bigg] \\
  \label{eq:infsup_Lqcf}
  &=	\frac{1}{2} \big( \phi''_F + 8 \phi''_{2F} \big).
\end{align}
Combining this estimate with Lemma~\ref{lem:infsup} gives the
following stability result.
\begin{theorem}
  Suppose that $\phi''_F + 8 \phi''_{2F} > 0$. Then the QCF system
  (\ref{qcf}) has a unique solution $\ub^{qcf}$, which satisfies
\begin{equation}\label{dirsta}
  \big\| D \ub^{qcf} \big\|_{\ell^\infty_\eps} \leq
  \frac{2 \|\fb\|_{*}}{\phi''_F+8\phi''_{2F}}
  + \Big|\frac{u_N^a - u_{-N}^a}{2N} \Big|,
\end{equation}
where
\begin{equation*}
\|{\bf f}\|_{*} := \sup_{\substack{\wb \in \Ds \\ \lpnorm{D \wb}{1} = 1}}
      \langle \fb, \wb \rangle.
\end{equation*}
\end{theorem}
\begin{proof}
  We write $\ub^{qcf} = \ub + \ub^D$ where $\ub \in \Ds$ and where
  $u^D_j = u_{-N}^a + (u_N^a - u_{-N}^a) (N+j)/(2N)$. Since $\ub^D$ is
  affine, it can be easily seen that $L^{qcf} \ub^D = 0$. Hence, the
  system is equivalent to $L^{qcf} \ub = \fb$. In view of
  \eqref{eq:infsup_Lqcf} and Lemma \ref{lem:infsup} this has a unique
  solution, and we have the stability bound
  \begin{displaymath}
    \| D \ub^{qcf} \|_{\ell^\infty_\eps} \leq \| D\ub \|_{\ell^\infty_\eps}
    + \| D\ub^D \|_{\ell^\infty_\eps} \leq \frac{2 \|\fb \|_*}{\phi_F''
      + 8\phi_{2F}''} + \Big|\frac{u_N^a - u_{-N}^a}{2N} \Big|. \qedhere
  \end{displaymath}
\end{proof}
\section{Convergence}
\label{sec:converge}

The quasicontinuum error $\eb^{qcf}=\ub^a - \ub^{qcf}$, where $\ub^a$
is again identified with its restriction to the computational domain
whenever necessary, satisfies the equation
\begin{equation*}
\begin{aligned}
  \left(L^{qcf}\eb^{qcf} \right)_j&=t_j,&\quad& j=-N+1,\dots,N-1,\\
  (\eb^{qcf})_{j}&=0,&\quad&j=-N,\,N.
\end{aligned}
\end{equation*}
Using (\ref{eq:qcf_const_1}) and (\ref{eq:poincare}) we see that the
truncation error ${\bf t}=L^{qcf}\ub^a-{\bf f}=(L^{qcf} - L^a) \ub^a$
(but with $t_N = t_{-N} = 0$) satisfies the negative norm estimate
\begin{displaymath}
\|{\bf t}\|_{*} = \sup_{\substack{\wb\in\Ds \\ \lpnorm{D \wb}{1} = 1}}
    \langle \, {\bf t},\,\wb \rangle \leq
    \sup_{\substack{\wb\in\Ds \\ \lpnorm{\wb}{\infty} = 1}}
    {\textstyle\frac12} \langle \, {\bf t},\,\wb \rangle = {\textstyle\frac12} \lpnorm{{\bf t}}{1}=
    {\textstyle \frac{1}{2}} \eps^2
  |\phi_{2F}''| \, \| \bar D^4 \ub^a \|_{\ell^1_\eps(\mathcal{C})}.
\end{displaymath}
However, we can get a slightly sharper result using the variational
representations of the operators $L^a$ and $L^{qcf}$ derived in
Section \ref{sec:conjugate}.

\begin{lemma}
  The truncation error satisfies the estimate
  \begin{displaymath}
    \| {\bf t} \|_* \leq 2 \eps^2 |\phi_{2F}''|
    \|D^3 \ub^a \|_{\ell^\infty_\eps(\widetilde{\mathcal{C}})},
  \end{displaymath}
  where $\widetilde{\mathcal{C}} = \{ -N+2, \dots, -K+1 \} \cup \{ K+2, \dots,
  N+1 \}$.
\end{lemma}
\begin{proof}
  Using the ``weak'' forms of $L^a$ and $L^{qcf}$ derived in
  (\ref{eq:La_weak}) and in Lemma \ref{th:ibp_lemma}, we obtain
  \begin{align*}
    \big\langle {\bf t}, \wb \big\rangle
    =~& \big\langle (L^{qcf} - L^a) \ub^a, \wb \big\rangle  \\
    =~& \phi_{2F}''\Bigg\{\sum_{j = -N+1}^{-K} \eps \big( - Du^a_{j-1} + 2Du^a_j - Du^a_{j+1} \big) D w_j  + \big( Du^a_{-K+1} - 2 Du^a_{-K} + Du^a_{-K-1} \big) w_{-K} \\
    &  + \sum_{j = K+1}^N \eps \big( - Du^a_{j-1} + 2Du^a_j - Du^a_{j+1} \big) Dw_j
    + \big( -Du^a_{K+2} + 2 Du^a_{K+1} - Du^a_{K} \big) w_{K} \Bigg\} \\
    \leq~& \eps^2 \| D^3\ub^a \|_{\ell^\infty_\eps(\widetilde{\mathcal{C}})} \big( \|D\wb \|_{\ell^1_\eps}
    + 2 \|\wb\|_{\ell^\infty_\eps} \big),
  \end{align*}
  where we used a weighted H\"{o}lder inequality in the last
  step. Using (\ref{eq:poincare}) to bound $\|\wb
  \|_{\ell^\infty_\eps}$ we obtain the stated bound.
\end{proof}

Combining this negative-norm truncation error estimate with the
stability estimate~\eqref{dirsta}, we obtain the following result.

\begin{theorem}
  Suppose that $\phi''_F + 8 \phi''_{2F} > 0.$ Then the atomistic
  problem (\ref{atom}) as well as the force-based quasicontinuum
  approximation (\ref{qcf}) have unique solutions, and they satisfy
  the error estimate
  \[
  \big\| D(\ub^a - \ub^{qcf}) \big\|_{\ell^\infty_\eps} \leq
  4 \eps^2\,\frac{|\phi_{2F}''| \| D^3 \ub^a \|_{\ell^\infty_\eps(\widetilde{\mathcal{C}})}}{
    \phi''_F+8\phi''_{2F}}.
  \]
\end{theorem}

As in Section~\ref{sec:model} we note again that it follows from the
interior regularity theory for elliptic finite difference
operators~\cite{thomee68} that $\| D^3 \ub^a
\|_{\ell^\infty_\eps(\widetilde{\mathcal{C}})}$ is bounded in the
continuum limit $\eps\to 0$, provided that $\fb$ is the restriction of
a smooth function in a neighborhood of the continuum region
$\widetilde{\mathcal{C}}$ to the lattice points.

\section{Estimates in Other Norms}\label{sec:norms}

We conclude this paper by showing that our choice of norms with
respect to which we analyzed the stability of the force-based QC
approximation was, in some sense, unique.
\begin{theorem}
  \label{th:noinfsup}
  Suppose that $\phi_F'' > 0,$ $\phi_{2F}'' \in \mathbb{R} \setminus
  \{0\},$ and that $1\leq p < \infty,$ and $1 < q \leq \infty$ so that
  $\frac{1}{p} + \frac{1}{q} = 1.$ Then there exists a constant $C>0$
  such that, for $2 \leq K \leq N /
  2$,
  \begin{equation*}
\inf_{\substack{\vb\in\Ds \\ \lpnorm{D \vb}{p} = 1}} \
\sup_{\substack{\wb\in\Ds \\ \lpnorm{D \wb}{q} = 1}}
    \langle  L^{qcf} \vb,\,\wb \rangle
    \le C N^{-1/p}.
    \end{equation*}
\end{theorem}
\begin{proof}
  We recall from Sections \ref{sec:conjugate} and \ref{sec:stability}
  that
\begin{equation*}
\inf_{\substack{\vb\in\Ds \\ \lpnorm{D \vb}{p} = 1}} \
\sup_{\substack{\wb\in\Ds \\ \lpnorm{D \wb}{q} = 1}}
    \langle  L^{qcf} \vb,\,\wb \rangle
=
\inf_{\substack{\xib \in \mvz \\ \lpnorm{\xib}{p} = 1}}
\sup_{\substack{\etab \in \mvz \\ \lpnorm{\etab}{q} = 1}}
    \langle  E^{qcf} \xib ,\,\etab\rangle
\leq
\inf_{\substack{\xib \in \mvz \\ \lpnorm{\xib}{p} = 1}}
   \big\| E^{qcf} \xib \big\|_{\ell^p_\eps},
\end{equation*}
where the second step follows from H\"older's inequality.  Therefore,
we obtain the stated result from the following lemma.
\end{proof}

\begin{lemma}
  Under the conditions of Theorem \ref{th:noinfsup} there exists a
  constant $C > 0$ such that, for $2 \leq K \leq N / 2,$
  \begin{equation*}
  \inf_{\substack{\xib \in \mvz \\ \lpnorm{\xib}{p} = 1}}
  \big\| E^{qcf} \xib \big\|_{\ell^p_\eps} \leq C N^{-1/p}.
   \end{equation*}
\end{lemma}
\begin{proof}
  The terms causing this effect are the nonlocal terms extending from
  the atomistic to continuum interface to the boundary.  Hence, we
  define
\begin{equation*}
\tilde{\xi}_j = \begin{cases}
-1, & j = -N+1,\dots,-K-1,\\
-\alpha, & j = -K, \\
\hphantom{-}0, & j = -K+1,\dots,K,\\
\hphantom{-}\alpha, & j = K+1, \\
\hphantom{-}1, & j = K+2,\dots,N,
\end{cases}
\end{equation*}
where $\alpha \in \Real$ will be specified below, and $\xib = \tilde
\xib / \|\tilde \xib \|_{\ell^p_\eps}$. Recalling the matrix
representation \eqref{eqcfd} for $E^{qcf}$, we see that
\begin{equation*}
E^{qcf} \tilde{\xib}
= \phi''_F \tilde{\xib} + \phi''_{2F}
\begin{cases}
-5 + 2\alpha, & j = -N+1,\dots,-K-1,\\
-1 - 2\alpha, & j = -K, \\
\hphantom{-12}-\alpha, & j = -K+1,\\
\hphantom{- 1 - 4}0, & j = -K+2,\dots,K-1,\\
\hphantom{- 1 - 5}\alpha, & j =  K,\\
\hphantom{-}1 + 2\alpha, & j =  K+1,\\
\hphantom{-}5 - 2\alpha, & j =  K+2,\dots,N,
\end{cases}
\end{equation*}
from which we obtain
\begin{equation*}
\begin{split}
\big\| E^{qcf} \tilde\xib \big\|_{\ell^p_\eps}^p
&= 2 \eps \Big(  |\alpha \phi''_{2F}|^p
+ |\alpha \phi''_F + (1+2\alpha)\phi''_{2F}|^p \\
&\qquad \qquad  + (N-K-1) |\phi''_F + (5-2\alpha)\phi''_{2F}|^p  \Big).
\end{split}
\end{equation*}

Choosing $\alpha = \frac{\phi''_F + 5 \phi''_{2F}}{2 \phi''_{2F}}$,
and thereby canceling the term $(N-K-1) |\phi''_F +
(5-2\alpha)\phi''_{2F}|^p$ above, gives
\begin{equation*}
\begin{split}
  \big\| E^{qcf} \tilde\xib \big\|_{\ell^p_\eps}^p
= 2 \eps \Big(\big| \alpha \phi''_{2F}\big|^p
 + \big| \alpha \phi''_F + (1+2\alpha) \phi''_{2F}\big|^p \Big).
\end{split}
\end{equation*}
Moreover, since
\begin{equation*}
  \big\| \tilde\xib \big\|_{\ell^p_\eps}^p
  = 2 \eps (N-K-1 + |\alpha|^p) \geq 2 \eps( N/2 - 1 + |\alpha|^p ),
\end{equation*}
we conclude that
\begin{equation*}
\inf_{\substack{\xib \in \mvz \\ \lpnorm{\xib}{p} = 1}}
  \big\| E^{qcf} \xib \big\|_{\ell^p_\eps}
\leq \left(\frac{\big| \alpha \phi''_{2F}\big|^p
 +  \big| \alpha \phi''_F + (1+2 \alpha) \phi''_{2F}\big|^p }
{N/2 - 1 + |\alpha|^p}\right)^{1/p}\leq C N^{-1/p}. \qedhere
\end{equation*}
\end{proof}

\section*{Conclusion}

We have presented a detailed stability and error analysis of the
force-based QC method in one dimension. Although we were able to
establish optimal order error estimates, we have also presented several
``negative'' results which are, in many respects, even more
interesting.  The present paper has focused exclusively on the force-based
QC method, but we expect that the lack of coercivity (and more generally
lack of stability in most norms) may be present in other force-based
coupling methods such as \cite{kohlhoff,shilkrot} or the QM-MM
coupling methods described in \cite{hybrid_review}.  A careful study
of these related methods is required to further understand and
establish force-based coupling techniques as predictive tools in
computational physics.

Finally, let us remark on the the fact that we have only proven
stability of the QCF method under the condition that $\phi_F'' + 8
\phi_{2F}'' > 0$. By contrast, the atomistic model is uniformly stable
if, {\em and only if} $\phi_F'' + 4 \phi_{2F}'' > 0$
\cite{sharpstability}. Hence, we expect that our above condition is
not sharp. Although we have established sharp characterizations of the
stability of other QC methods in \cite{sharpstability}, we were unable
to rigorously achieve the same for the force-based QC method as well.
While our computational results reported in \cite{sharpstability} do
indicate that $L^{qcf}$ has only positive eigenvalues if $\phi_F'' + 4
\phi_{2F}'' > 0,$ this is not enough to ensure stability, uniformly as
$N \to \infty$, in the regime $\phi_F'' + 4\phi_{2F}'' > 0$ for the
$W^{1,\infty}$-$W^{1,1}$ ``duality pairing'' or for any other norm
with a recognizable continuum limit.

}
\end{document}